\documentclass{article}
\usepackage{latexsym}
\usepackage[all]{xy}
\usepackage{amssymb,amsmath,amsthm,verbatim,color}
\usepackage{enumitem}

\def\Sym{\mathrm{Sym}}

\def\Hom{\mathrm{Hom}}

\def\Ext{\mathrm{Ext}}

\def\TC{T\langle C\rangle}
\def\olTC{T\langle\ol C\rangle}
\def\wt{\widetilde}

\def\bP{\mathbb{P}}

\def\lra{\longrightarrow}

\def\cO{\mathcal{O}}

\def\F{\mathcal{F}}
\def\A{\mathcal{A}}

\def\J{\mathcal{J}}

\def\I{\mathcal{I}}

\def\E{\mathcal{E}}

\def\T{\mathcal{T}}
\def\C{\mathcal{C}}
\def\M{\mathcal{M}}

\def\V{\mathcal{V}}

\def\TC{T\langle C\rangle}
\def\dim{\mathrm{dim}}

\def\ol{\overline}
\def\wt{\widetilde}

  \newtheorem{theorem}{Theorem}[section]
\newtheorem{lemma}[theorem]{Lemma}
\newtheorem{prop}[theorem]{Proposition}

\newtheorem{corollary}[theorem]{Corollary}

\title{IVHS of nodal plane curves}
\author{Edoardo Sernesi }
\date{}

\begin{document}

\maketitle
\begin{abstract}
   Let $\V_{d,n}$ be the Severi variety of irreducible plane curves of degree  $d\ge 4$ having $n$ nodes, with $0\le n \le \binom{d-1}{2}-1$.  We prove that for every   $[\ol C]\in \V_{d,n}$, the infinitesimal variation of the Hodge structure of the normalization $C$ of $\ol C$ is  maximal as $[\ol C]$  moves in $\V_{d,n}$. As a preliminary result, we also prove that the family of curves of genus $g \ge 1$ mapping with degree $d \ge 2$ to a fixed curve $Y$ of genus $\pi$ has maximal variation if and only if $\pi = 0$.
\end{abstract}

\section{Introduction}
All schemes in this paper will be defined over $\mathbb C$, the field of complex numbers.

Let $g \ge 1$ and let $\M_g$, resp. $\A_g$, be the stack of smooth projective curves of genus $g$, respectively the stack of principally polarized abelian varieties of dimension $g$. The \emph{Torelli map}
$$
\T: \M_g \lra \A_g
$$
assigns to a curve $[C]\in \M_g$ its   polarized jacobian $\T([C])=(J(C),\Theta)$ where $\Theta$ is a symmetric theta divisor. At a given $[C]$ the differential of this morphism is a linear map
$$
d\T_{[C]}: H^1(C,T_C) \lra \Sym^2H^1(C,\cO_C)
$$
(see e.g. \cite[Example 3.4.24(iii)]{eS06}) whose codomain we will view as the subset
$$
\Hom_{\frak S}(H^0(C,\omega_C),H^1(C,\cO_C))\subset \Hom(H^0(C,\omega_C),H^1(C,\cO_C))
$$
consisting of self-dual linear maps. $d\T_{[C]}$ can be described in more than one way. Firstly, following \cite{OS79} and \cite{pG83}, we can interpret  $d\T_{[C]}$ as the dual of  the natural multiplication map
$$
\mu: \Sym ^2 H^0(C,\omega_C)\lra H^0(C,\omega_C^{\otimes 2})
$$
 Alternatively, one can describe $d\T_{[C]}$ by representing each element $v\in H^1(C,T_C)=\Ext^1(\omega_C,\cO_C)$ by an extension 
$$
v: 0 \lra \cO_C \lra \E \lra \omega_C \lra 0
$$
and letting $d\T_{[C]}(v)=\delta_v$, where
$$
\delta_v:H^0(C,\omega_C)\lra H^1(C,\cO_C)
$$
is the coboundary map of $v$. It is straightforward to check that the two descriptions of 
$d\T_{[C]}$ coincide.   The map $d\T_{[C]}(v)=\delta_v$ is called \emph{infinitesimal variation of the Hodge structure} (IVHS) of $C$ along $v$, and the map $d\T_{[C]}$ is also called \emph{IVHS map}.

\emph{Torelli's Theorem} asserts that $\T$ is injective. Nevertheless $\T$ is not everywhere an embedding: the \emph{Infinitesimal Torelli's Theorem} asserts that $\T$ is an embedding at $[C]$, i.e. $d\T_{[C]}$ is injective, if and only if  $g \le 2$ or $g \ge 3$ and $C$ is non-hyperelliptic.   Using the first interpretation of $d\T_{[C]}$ this fact is an immediate consequence of the classical \emph{Noether's Theorem}, which asserts that $\mu$ is surjective if and only if $g \le 2$ or $g\ge 3$ and $C$ is non-hyperelliptic.

Observing that the IVHS map $d\T_{[C]}(v)=\delta_v$ is a linear map between two vector spaces (of the same dimension $g$), we can associate to $v$ the rank of $\delta_v$, which is called \emph{variation} of $v$ in the literature. One says that $v$ has \emph{maximal variation} if $\delta_v$ is an isomorphism, i.e. if $v$ has variation equal to $g$.  This notion should not be confused with the one considered in \cite{DH21} and in \cite{aB25} with a different meaning.

It is natural to ask what are the possibilities for the variation of $v$, i.e. for the rank of   $\delta_v$, in particular whether it can be maximal, as $C$ and $v\in H^1(C,T_C)$ vary under suitable given constrains. 
 In the most general situation, namely when $C$ and $v$ are arbitrary, we have the following
 
 \begin{prop}[\cite{pG83}, Corollary (5.20)]\label{P:vararbcur}
    For any curve $C$ of genus $g \ge 1$, there is  $v\in H^1(C,T_C)$ having maximal variation.
 \end{prop}

For a different proof see also \cite{FP21}, Lemma 1.3. 
 Given the above, the next natural questions arise if we assume that $C$ is not arbitrary but appears as the fibre $C\cong \C(b)$ over a point $b\in B$ of a  family $f:\C \lra B$ of projective nonsingular curves of genus $g$. We then consider the Kodaira-Spencer map
$$
\kappa_b(f): T_bB \lra H^1(C,T_C)
$$
and the composition
\begin{equation}
    T_bB \xrightarrow{\kappa_b(f)}H^1(C,T_C)\xrightarrow{d\T_{[C]}}\Hom_{\frak S}(H^0(C,\omega_C),H^1(C,\cO_C))
\end{equation}
which we call  \emph{IVHS map of $f$ at $b$}.
  We can ask about the variation of $v\in \mathrm{Im}(\kappa_b(f))$. If $\kappa_b(f)$ is not surjective (meaning that the family $f$ has \emph{special moduli)} then it is not at all clear  what are the possible values of the variation of $v\in \mathrm{Im}(\kappa_b(f))$, nor if there exists such a $v$ with maximal variation. If it exists we say that $f$ \emph{has maximal variation at} $b$, or that $C$ \emph{has maximal variation as a member of} $f$. Similarly we will call   \emph{minimal variation} of $f$ at $b$ to be the minimal value of the variation  of $v$ among all $0\ne v\in \mathrm{Im}(\kappa_b(f))$. 

  Typical examples   one wants to consider are linear systems of curves on algebraic surfaces. 
  Let $S$ be a nonsingular algebraic surface,   $L$ a line bundle on $S$ and $B\subset |L|$   the open subset parametrizing   nonsingular connected curves. Then $f:\C \lra B$  is the family of curves parametrized by $B$, and   $C:=\C_b$ for some $b\in B$. 
We have the following result by Favale-Pirola and by Favale-Naranjo-Pirola in the case $S=\bP^2$ and $L=\cO_{\bP^2}(d)$.

\begin{theorem}\label{T:FPintro}
\begin{itemize}
    \item[(i)]\cite{FP21} Let $C\subset \bP^2$ be a nonsingular curve of degree $d \ge 4$. Then $C$ has maximal variation as a member of  the family of nonsingular plane curves of degree $d$.
    \item[(ii)]\cite{FNP17} The minimal possible variation for a plane curve of degree $d\ge 5$ is $d-3$, and the minimum is attained by the Fermat curve.
\end{itemize}
\end{theorem}

The case of $S=\bP^1\times \bP^1$ has   been studied by Gonzalez-Alonso and Torelli. Their result is the following:

\begin{theorem}[\cite{GT24}]\label{T:GAintro}
    Let $C$ be a nonsingular ample curve in $\bP^1\times \bP^1$. Then the family of deformations of $C$ in $|C|$ has maximal variation at $[C]$. 
\end{theorem}

Other interesting families to be considered are  families  of morphisms.  Given    a projective nonsingular variety $Y$ and a non-constant $\varphi:C \lra Y$, where $[C]\in \M_g$, there is a local universal family:
\begin{equation}\label{E:locunivfam}
    \xymatrix{
\C_B \ar[dr]_-f\ar[rr]^-\Phi&&B\times Y \ar[dl]\\
&B
}
\end{equation}
of deformations of $\varphi$,
uniquely defined up to etale base change (see \cite{DS17}).
The IVHS of the  family $f$ of curves of genus $g$ appearing in this diagram at the point $b\in B$ parametrizing $\varphi$ will be called, with abuse of language,  \emph{ IVHS of $\varphi$}, and we will talk about \emph{variation of $\varphi$} to mean the variation of $f$ at $b$. We want to compute such variation  in various cases. The first   case we consider is  when $Y$ is a curve. This is our result.

\begin{theorem}\label{T:Ycurve}
    Let $\varphi:C \lra Y$ be a  morphism of degree $d\ge 2$ from a nonsingular projective connected curve of genus $g>0$ to a nonsingular curve of genus $\pi$. Then $\varphi$ has maximal variation if and only if $\pi=0$.
    \end{theorem}

Using completely different (and quite interesting) methods Theorem \ref{T:Ycurve}  has been proved  in \cite{jZ25} for  trigonal curves.

\noindent
 We then turn to maps with target $Y=\bP^2$.
Our   result in this direction is the following one:

\begin{theorem}\label{T:nodalintro}
   Let $\ol C\subset \bP^2$ be an irreducible   curve of degree $d \ge 4$ and geometric genus $g \ge 1$ having only nodes (ordinary double points) as singularities. Then the map $\varphi:C \lra \bP^2$ induced by the normalization of $\ol C$ has maximal variation.
\end{theorem}

\noindent
 This theorem contains part (i) of Theorem \ref{T:FPintro} as a special case, but our proof is  different from the one given by Favale and Pirola.  It would be very interesting to extend Theorem \ref{T:nodalintro} to curves with more general singularities. Unfortunately our method of proof does not seem to apply   to curves having singularities different from nodes.

 The paper is organized as follows. In \S 2 we give a cohomological approach to IVHS of curves in families in various situations, and we recall some results, e.g. the non maximality of IVHS of curves on irregular surfaces. We also prove Theorem \ref{T:Ycurve}.
 In \S 3 we discuss families of nodal curves on surfaces, and in \S 4 we specialize to the case of $\bP^2$. Theorem \ref{T:nodalintro} is proved in the final \S 5.

 \medskip
{\bf Acknowledgements.}  For interesting conversations I am grateful to Andrea Bruno, Andreas Knutsen, Margherita Lelli Chiesa, Riccardo Salvati Manni. I warmly thank Victor Gonzalez-Alonso and  Sara Torelli for a careful reading of a first draft of this work.

\section{Cohomological set-up}\label{S:cohomol}

We fix  a non-constant morphism $\varphi: C \lra Y$   from a projective nonsingular curve $C$ of genus $g \ge 1$ to a nonsingular projective variety $Y$. We let \eqref{E:locunivfam} be the local universal family of deformations of $\varphi$. We have  an exact sequence  on $C$:
\begin{equation}\label{E:normphi}
    0 \lra T_C \lra \varphi^*T_Y \lra N_\varphi \lra 0
\end{equation}
whose coboundary map 
$$
\kappa_{\varphi}:H^0(C,N_\varphi)\lra H^1(C,T_C)
$$
is the Kodaira-Spencer map at $[\varphi]$ of  the family $f$ appearing in \eqref{E:locunivfam}. Composing with $d\T_{[C]}$ we get the IVHS map
\begin{equation}\label{E:IVHSphi} H^0(C,N_\varphi)\xrightarrow{\kappa_\varphi} H^1(C,T_C)\xrightarrow{d\T_{[C]}}\Hom_{\frak S}(H^0(C,\omega_C),H^1(C,\cO_C))
\end{equation}
whose maximal variation we want to compute. 
 For this purpose it is convenient to tensor \eqref{E:normphi} by $\omega_C$:
 \begin{equation*}
    \zeta_\varphi: 0 \lra \cO_C \lra \varphi^*T_Y\otimes\omega_C \lra N_\varphi\otimes\omega_C \lra 0
 \end{equation*}
For each $\sigma\in H^0(C,N_\varphi)$     the pullback of $\zeta_\varphi$ by $\sigma$ is $\kappa_\varphi(\sigma)\in H^1(C,T_C)$:
\begin{equation}\label{E:pullback}
    \xymatrix{
    \sigma^*(\zeta_\varphi)=\kappa_\varphi(\sigma):&0\ar[r]&\cO_C\ar@{=}[d]\ar[r]&\E\ar[d]\ar[r]&\omega_C\ar[r]\ar[d]^-\sigma&0 \\
    \zeta_\varphi:&0\ar[r]&\cO_C \ar[r]&\varphi^*T_Y\otimes\omega_C\ar[r]^-q&N_\varphi\otimes\omega_C\ar[r]&0
    }
\end{equation}
and therefore $d\T_{[C]}(\kappa_\varphi(\sigma))$ is the coboundary map of $\sigma^*(\zeta_\varphi)$, and its rank is the variation of $\sigma$.  The following lemma is obvious.

\begin{lemma}\label{L:IVHSphi}
   In the above situation, $\sigma$ has maximal variation if and only if $H^0(\sigma)$ is injective and
   \begin{equation}\label{E:inters}
       \mathrm{Im}(H^0(q))\cap \mathrm{Im}(H^0(\sigma))=(0)
   \end{equation}
\end{lemma}


 This lemma  already suffices to prove Theorem \ref{T:Ycurve} in the  case $\dim(Y)=1$.

\medskip 

{\bf Proof of Theorem \ref{T:Ycurve}}.

\noindent 
Let    $\pi$ be the genus of $Y$, and $d=\deg(\varphi)$; we have the following expression for the degree  of the ramification divisor $R$ of $\varphi$: 
\begin{equation}\label{E:degR}
\deg(R)=\deg(\varphi^*T_Y\otimes\omega_C)=2[g-1+(1-\pi)d]
\end{equation}
Since $Y$ is a curve  $N_\varphi$ is torsion, hence $H^1(N_\varphi\otimes\omega_C)=0$. If $\pi>0$ then $H^1(C,\varphi^*T_Y\otimes\omega_C)\ne (0)$, and the coboundary of $\zeta_\varphi$ cannot be surjective. Therefore $\varphi$ does not have maximal variation: this proves the ``only if'' part of the theorem. 

Now assume $Y=\bP^1$ and let $L=\varphi^*\cO_{\bP^1}(1)$. Then $\deg(R)=2(d+g-1)$,  
$$
h^0(C,\varphi^*T_Y\otimes\omega_C)=h^0(C,\omega_CL^2)=2d+g-1, \quad H^1(C,\varphi^*T_Y\otimes\omega_C)=0
$$
and rk$(H^0(q))=2d+g-2$. Keeping  Lemma \ref{L:IVHSphi} in mind, we need to find $\sigma\in H^0(C,N_\varphi)$ such that \eqref{E:inters} holds and $H^0(\sigma)$ is injective. Choose an effective subdivisor $R_1<R$ of $\deg(R_1)=g$ such that $H^0(C,\omega_C(-R_1))=0$. Let $R_2=R-R_1$. Then $\deg(R_2)=2d+g-2$ and $h^0(C,\omega_CL^2(-R_2))=1=h^0(C,\omega_CL^2(-R))$. Therefore 
$$
\ker(H^0(q))= H^0(C,\omega_CL^2(-R))=H^0(C,\omega_CL^2(-R_2))
$$
We choose $\sigma=1 R_1+0 R_2$.  Then every nonzero element of $\mathrm{Im}[H^0(\sigma)]$ is zero on $R_2$ and therefore cannot be in $\mathrm{Im}(H^0(q))$. Moreover  $H^0(\sigma)$ 
 is injective by the choice of $R_1$. By Lemma \ref{L:IVHSphi} we see that $\sigma$ has maximal variation. 
\hfill$\square$
 
\medskip 
We now \emph{suppose that $Y$ is a nonsingular surface} and that $\varphi:C \lra Y$ is an immersion, i.e. that its differential is everywhere injective. Then $N_\varphi$ is   invertible and, to avoid trivial cases, we may assume that it satisfies $H^0(C,N_\varphi))\ne 0$. Therefore $N_\varphi\otimes\omega_C$ is nonspecial. From   \eqref{E:pullback} it follows that if $H^1(C,\varphi^*T_Y\otimes\omega_C)\ne 0$ then the coboundary map of $\zeta_\varphi$ cannot be surjective. We thus have the following:

\begin{prop}\label{P:necescond}
    In the above situation a necessary condition for $\varphi$ to have maximal variation is that $H^1(C,\varphi^*T_Y\otimes\omega_C)=0$ or, equivalently, that $H^0(C,\varphi^*\Omega^1_Y)=0$.
\end{prop}

The following corollary is in \cite{FP21}, Prop. 1.7:

\begin{corollary}\label{C:irregular}
  Let $C$ be an ample nonsingular curve in an algebraic surface $Y$ such that  $H^0(Y,\Omega^1_Y)\ne 0$. Then $C$ does not have maximal variation in $|C|$.  
\end{corollary}

\begin{proof}
    Consider the exact sequence
      $$
      0\lra T_Y\otimes\omega_Y\lra T_Y\otimes\omega_Y(C) \lra T_{Y|C}\otimes\omega_C\lra 0
      $$
      By Serre duality we have
      $$
      H^2(Y,T_Y\otimes\omega_Y(C))\cong H^0(Y,\Omega^1_Y(-C))
      $$
      and this is zero by Nakano's vanishing theorem. Therefore we have a surjection
      $$
       H^1(C,T_{Y|C}\otimes\omega_C)\lra H^2(Y,T_Y\otimes\omega_Y)\cong H^0(Y,\Omega^1_Y)\ne 0
      $$
thus  $H^1(C,T_{Y|C}\otimes\omega_C)\ne (0)$.     From Proposition \ref{P:necescond} we deduce that $C$ does not have maximal variation in $|C|$.
\end{proof}

When $Y$ is a surface and we are assuming that $N_\varphi$ is invertible and 
$$
H^1(C,T_{Y|C}\otimes\omega_C)=0
$$
then $H^0(\sigma)$ is injective and a simple computation shows that
   $$ h^0(C,N_\varphi\otimes\omega_C)=h^0(C,\varphi^*T_Y\otimes\omega_C)-1+g = \mathrm{rk}(H^0(q))+\mathrm{rk}(H^0(\sigma))
   $$
   Therefore $\mathrm{Im}(H^0(q))$ and $\mathrm{Im}(H^0(\sigma))$ have complementary dimensions inside $H^0(C,N_\varphi\otimes\omega_C)$.

\noindent  
  When $Y$ is a regular surface the condition of Proposition \ref{P:necescond} is often satisfied. For example  nonsingular     curves in $\bP^2$ and in   $\bP^1\times\bP^1$ satisfy, as implied by theorems \ref{T:FPintro} and \ref{T:GAintro} and Corollary \ref{C:irregular}. 

  \medskip
  
  {\bf Curves on K3 surfaces.}
  Let $C$ be a very ample nonsingular curve on a K3 surface $S$. If $\Omega^1_{S|C}$ is stable then $h^1(C,\T_{S|C} \otimes\omega_C)=h^0(C,\Omega^1_{S|C})=0$,   
 In \cite{DH21}, Theorem 2.9, the authors give   conditions for $h^0(C,\Omega^1_{S|C})=0$,  weaker than the stability of $\Omega^1_{S|C}$. I do  not know examples of such $C\subset S$ having maximal variation.
 Examples of nonsingular ample curves  $C$ on a K3 surface $S$ such that $\Omega^1_{S|C}$ is not stable  have been given in \cite{GO20}, but $h^0(C,\Omega^1_{S|C})$ has not been computed for such curves.

 Assume that   
    \begin{equation}\label{E:rmK3}
        h^1(C,T_{S|C}\otimes\omega_C)>0
    \end{equation}
    Then Proposition \ref{P:necescond} predicts that $C$ does not have maximal variation in $S$. But we can say more. We have
    $$
0 < h^1(C,T_{S|C}\otimes\omega_C)=h^0(C,\Omega^1_{S|C})=h^0(C,T_{S|C})
$$

Let $0\ne \tau:\cO_C\to T_{S|C}$ be a section of $T_{S|C}$ and $\sigma:\cO_C\xrightarrow{\tau}T_{S|C}\to \omega_C$. The diagram
$$
\xymatrix{
&0\ar[r]&T_C\ar[r]\ar@{=}[d]&T_{S|C}\ar[r]&\omega_C\ar[r]&0\\
\kappa_\varphi(\sigma):&0\ar[r]&T_C\ar[r]&\E\ar[r]\ar[u]&\cO_C\ar[ul]^-\tau\ar[r]\ar[u]^-\sigma&0
}
$$
shows that $\kappa_\varphi(\sigma)\in H^1(C,T_C)$ splits. Since   $\sigma\ne 0$, the Kodaira-Spencer map $\kappa_\varphi:H^0(C,\omega_C) \lra H^1(C,T_C)$ is not injective.

   \section{Nodal curves on surfaces}\label{S:cohomonodal} 
   Let $Y$ be a nonsingular algebraic surface, $C$ a projective nonsingular connected curve,  and    $\varphi: C \lra Y$   a morphism  birational onto its image $\ol C$.   We denote by $\omega_{\ol C}$  the dualizing sheaf of $\ol C$, and by  $$
   \mathfrak{c}:=Hom_{\ol C}(\varphi_*\cO_C,\cO_{\ol C})\subset \cO_{\ol C}
   $$ 
    the conductor ideal sheaf. Recall   that 
   \begin{equation}\label{E:pushdomega}
       \varphi_*\omega_C = Hom_{\ol C}(\varphi_*\cO_C,\omega_{\ol C})= \mathfrak{c}\otimes\omega_{\ol C}
   \end{equation}
   or, equivalently, that
   $$
   \mathfrak c = \varphi_*\omega_C\otimes \omega_{\ol C}^{-1}
   $$
  In fact $\omega_C=\varphi^!\omega_{\ol C}:=Hom_{\ol C}(\varphi_*\cO_C,\omega_{\ol C})$ (\cite[Ex. III.7.2, p. 249]{rH77}) and
  $$
  \varphi_*\omega_C = \varphi_*Hom_C(\cO_C,\omega_C)=\varphi_*Hom_C(\cO_C,\varphi^!\omega_{\ol C})=Hom_{\ol C}(\varphi_*\cO_C,\omega_{\ol C})
  $$
     (see \cite[Ex. III.6.10 p. 239]{rH77}).

     \noindent
     We assume from now on in this section that \emph{$\ol C$ has only ordinary double points (nodes) as singularities}. We denote 
  by   $\I_\Delta\subset \cO_Y$  the  ideal sheaf of the reduced scheme $\Delta$ of nodes of $\ol C$.  The restriction of $\I_\Delta$ to $\ol C$ is the conductor $\mathfrak c$ and $\I_\Delta$ is the adjoint ideal sheaf of $\ol C$ (see e.g. \cite{lS79}).
  
\begin{lemma}\label{L:pushdnormal}
  Dualizing the conormal sequence
   $$
   0 \lra \cO_{\ol C}(-\ol C) \lra \Omega^1_{Y|\ol C}\lra \Omega^1_{\ol C}\lra 0
   $$
   we obtain the sequence:
   \begin{equation}\label{E:pushdnormal}
    0\lra \varphi_*\cO_C\otimes\omega_{\ol C}^{-1}\lra T_{Y|\ol C} \lra \varphi_*N_\varphi\lra 0   
   \end{equation}
   \end{lemma}
\begin{proof}
Dualizing the conormal sequence we get:
$$
0 \lra Hom(\Omega^1_{\ol C},\cO_{\ol C}) \lra T_{Y|\ol C}\xrightarrow{\partial} \cO_{\ol C}(\ol C) \lra Ext^1(\Omega^1_{\ol C},\cO_{\ol C})\lra 0
$$
and  
$$
\mathrm{Im}(\partial)=\mathfrak{c}\otimes\cO_{\ol C}(\ol C)
$$
since the map $\partial$ is locally defined by the partial derivatives of a local equation of $\ol C$. \footnote{In \cite[p. 174]{eS06} the sheaf $\mathrm{Im}(\partial)$ is denoted by $N'_{\ol C}$, and called \emph{equisingular normal sheaf} of $\ol C$.} 
Moreover (\cite[Lemma 1.5]{aT84}, see  also \cite[Lemma 3.4.15]{eS06}):
$$
\mathrm{Im}(\partial)=\varphi_*N_\varphi
$$
 On the other hand, from the exact sequence
 $$
 0\lra \tau \lra \Omega^1_{\ol C} \lra \mathfrak c \otimes  \omega_{\ol C} \lra 0
 $$
 where $\tau\subset\Omega^1_{\ol C}$ is the torsion subsheaf (see \cite[Corollary 4.13]{qL02} and \cite[Example B.9(ii)]{eS06} we deduce that:
\begin{align*}
 T_{\ol C}=Hom(\Omega^1_{\ol C},\cO_{\ol C})& = Hom(Hom(\varphi_*\cO_C,\cO_{\ol C})\otimes\omega_{\ol C},\cO_{\ol C})\\
 &=Hom(Hom(\varphi_*\cO_C,\cO_{\ol C}),\cO_{\ol C})\otimes\omega_{\ol C}^{-1}\\
 &=\varphi_*\cO_C\otimes\omega_{\ol C}^{-1}
\end{align*}

Therefore we get \eqref{E:pushdnormal}.
\end{proof}
The analysis carried out in \S \ref{S:cohomol} can be repeated  word by word in this case.  Let 
$$
\sigma\in H^0(C,N_\varphi)=H^0(\ol C,\varphi_*N_\varphi)= \mathfrak c\otimes \cO_{\ol C}(\ol C)
$$  
The IVHS of $C$ associated to $\kappa_\varphi(\sigma)\in H^1(C,T_C)$ is the coboundary map of the upper sequence of the following  pullback diagram analogous to \eqref{E:pullback}:
\begin{equation}\label{E:pullbacksing}
    \xymatrix{
  0\ar[r]&\varphi_*\cO_C\ar[r]\ar@{=}[d]&\E'\ar[r]\ar[d]&\varphi_*\omega_C\ar[d]^-\sigma\ar[r]&0 \\
 0\ar[r]&\varphi_*\cO_C\ar[r]&T_{Y|\ol C}\otimes\omega_{\ol C}\ar[r]&\mathfrak{c}\otimes\omega_{\ol C}(\ol C)\ar[r]&0 
    }
\end{equation}
where the second row is \eqref{E:pushdnormal} tensored by $\omega_{\ol C}$. Unfortunately this diagram is not directly helpful in computing the  variation of $\kappa_\varphi(\sigma)$. In the next section we will switch to plane curves and we will add  the sheaf of logarithmic vector fields to our arsenal.


   \section{IVHS of plane nodal curves}

    We fix   an irreducible plane curve $\ol C: F(X,Y,Z)=0$ of degree $d \ge 4$ having at most nodal singularities (ordinary double points). We let   $\varphi:C \lra \bP^2$ be the morphism induced by the normalization map of $\ol C$ and assume that $C$ has genus $g \ge 1$.  We introduce the (locally free) sheaf of vector fields on $\bP^2$ which are tangent to $\ol C$. It is defined by the following commutative diagram with exact  rows:
  \begin{equation}\label{E:keynodal1}
   \xymatrix{
   &0\ar[d]&0\ar[d]\\
   &T_{\bP^2}(-d)\ar[d]\ar@{=}[r]&T_{\bP^2}(-d)\ar[d]\\
   0 \ar[r]& \olTC \ar[d]\ar[r]& T_{\bP^2}\ar[d] \ar[r]& \varphi_*N_\varphi\ar[r]\ar@{=}[d]& 0 \\
   0 \ar[r]&\varphi_*\cO_C\otimes\omega_{\ol C}^{-1}\ar[d] \ar[r]&T_{\bP^2|\ol C}\ar[r]\ar[d]&\varphi_*N_\varphi\ar[r]&0 \\
   &0&0
   }
   \end{equation}
   The middle row shows that:
   \begin{equation*}
      H^1(\bP^2,\olTC)\cong H^0(C,N_\varphi)/H^0(\bP^2,T_{\bP^2}) 
   \end{equation*}
   Since $\omega_{\ol C}=\cO_{\ol C}(d-3)$ and $\varphi_*N_\varphi=\mathfrak c\otimes\cO_C(d)$, tensoring by $\cO_{\bP^2}(d-3)$ we obtain:
   \begin{equation}\label{E:keynodal2}
   \xymatrix{
   &0\ar[d]&0\ar[d]\\
   &T_{\bP^2}(-3)\ar[d]\ar@{=}[r]&T_{\bP^2}(-3)\ar[d]\\
   0 \ar[r]& \olTC(d-3) \ar[d]\ar[r]& T_{\bP^2}(d-3)\ar[d] \ar[r]& \mathfrak c \otimes\omega_{\ol C}(d)\ar[r]\ar@{=}[d]& 0 \\
   0 \ar[r]&\varphi_*\cO_C\ar[d] \ar[r]&T_{\bP^2|\ol C}\otimes\omega_{\ol C}\ar[r]\ar[d]&\mathfrak c \otimes\omega_{\ol C}(d)\ar[r]&0 \\
   &0&0
   }
   \end{equation}

\begin{lemma}\label{L:TC1}
The left vertical sequence of \eqref{E:keynodal2} induces an  isomorphism:
 $$
 H^1(\bP^2,\olTC(d-3))\cong H^1(C,\cO_C)
 $$
\end{lemma}
\begin{proof} 
  We have $H^2(\bP^2,T_{\bP^2}(-3))=(0)$. Therefore the map
  $$
  H^1(\bP^2,\olTC(d-3))\lra H^1(\ol C,\varphi_*\cO_C)=H^1(C,\cO_C)
  $$
  is surjective. But since $h^1(\bP^2,\olTC(d-3))=h^1(C,\cO_C)$  \cite[ Prop. 3.1]{DS14}, the map is an isomorphism.
\end{proof}

   
 Consider the commutative and exact diagram:
 \begin{equation}\label{E:keydiag2}
    \xymatrix{
    0\ar[r]&\olTC\ar@{=}[d]\ar[r]& \cO_{\bP^2}(1)^3\ar[d]\ar[r]^-\partial&\I_\Delta(d)\ar[r]\ar[d]&0
    \\
0\ar[r]&\olTC\ar[r]&T_{\bP^2}\ar[r]&\mathfrak{c}\otimes\cO_{\ol C}(d)\ar[r]&0
}
\end{equation}
where $\I_\Delta=\wt{J_F}$ is the adjoint ideal sheaf,  which is defined by the jacobian homogeneous ideal  $J_F=(F_X, F_Y, F_Z)$ and $\partial$ is defined by $F_X, F_Y, F_Z$.  The restriction of $\I_\Delta(d)$ to $\ol C$ is   $\mathfrak{c}\otimes\cO_{\ol C}(d)=\varphi_*N_\varphi$, and the map 
$$
H^0(\bP^2,\I_\Delta(d)) \lra H^0(\ol C,\varphi_*N_\varphi), \quad \wt \sigma\mapsto \sigma
$$
is surjective.

 Twisting by $\cO_{\bP^2}(d-3)$, we  replace \eqref{E:keynodal2} by
\begin{equation*}
  \xymatrix{
    0\ar[r]&\olTC(d-3)\ar@{=}[d]\ar[r]& \cO_{\bP^2}(d-2)^3\ar[d]\ar[r]^-\partial&\I_\Delta(2d-3)\ar[r]\ar[d]&0
    \\
0 \ar[r]&\varphi_*\cO_C \ar[r]&T_{\bP^2|\ol C}\otimes\omega_{\ol C}\ar[r]&\mathfrak c \otimes\omega_{\ol C}(d)\ar[r]&0
}  
\end{equation*}
    We can associate to every $\sigma\in H^0(\ol C, \varphi_*N_\varphi)$ the following pullback diagram:
   \begin{equation}\label{E:pullbacksing2}
    \xymatrix{
    0\ar[r]&\olTC(d-3)\ar@{=}[d] \ar[r]&\F\ar[r]\ar[d]&\I_\Delta(d-3)\ar[r]\ar[d]^-{\wt \sigma}&0 \\
    0\ar[r]&\olTC(d-3) \ar[r]& \cO_{\bP^2}(d-2)^3 \ar[r]^-\partial&\I_\Delta(2d-3)\ar[r]&0
    }   
   \end{equation}
where  $\wt \sigma\in H^0(\bP^2,\I_\Delta(d))$ and $\wt \sigma\mapsto \sigma$.
Recalling Lemma \ref{L:TC1} and that the restriction   to $\ol C$ induces an isomorphism
$$
 H^0(\bP^2,\I_\Delta(d-3)) \cong H^0(\ol C,\varphi_*\omega_C)
$$
we see that the coboundary map of the upper sequence is the IVHS of $C$ associated to $\kappa_\varphi(\sigma)$. We will use this diagram to prove Theorem \ref{T:nodalintro}.


\section{Proof of Theorem \ref{T:nodalintro}}

 From the discussion of the previous section it follows that Theorem \ref{T:nodalintro} is a consequence of the following:

 \begin{prop}\label{T:ffinal}
     Let $\ol C\subset \bP^2$ be a nodal irreducible curve of degree $d\ge 3$ with $n\le \binom{d-1}{2}-1$ nodes and no other singularities and denote by $\I_\Delta\subset \cO_{\bP^2}$   the ideal sheaf of the scheme $\Delta\subset\bP^2$ of nodes of $\ol C$.  Let $\wt \sigma\in H^0(\bP^2,\I_\Delta(d))$ be a general element.
     Then the coboundary map of the upper sequence of diagram \eqref{E:pullbacksing2} is an  isomorphism
$$
H^0(\bP^2,\I_{\Delta}(d-3)) \xrightarrow{\cong} H^1(\bP^2,\olTC(d-3)) 
$$
 \end{prop}

For the proof of Proposition \ref{T:ffinal} we need a lemma.

\begin{lemma}\label{L:Zexists2}
    In the situation of Proposition \ref{T:ffinal}, let $F_1,F_2,F_3$ be a general basis of $J_{F,d-1}=\langle F_X,F_Y,F_Z\rangle$. Denote by the same symbols the curves they define. Then for each $i\ne j$ we have:
    \begin{itemize}
        \item[(I)] $F_i \cap F_j=\Delta\cup\Sigma_{ij}$, where $\Sigma_{ij}$ consists of $(d-1)^2-n$ distinct points and $\Sigma_{ij}\cap \Delta = \emptyset$.
        \item[(II)] For fixed $i\ne j$ it is possible to decompose $\Sigma_{ij}= Z\cup Y$, with $Z\cap Y=\emptyset$, $\ell(Z)=\binom{d}{2}-n$, $\ell(Y)=\binom{d-1}{2}$, and
        $$
        h^0(\bP^2,\I_{Z\cup \Delta}(d-2))=0=h^0(\bP^2,\I_{Y}(d-3))
        $$
    \end{itemize}
\end{lemma}

\begin{proof}
    (I) The linear system $|J_{F,d-1}|$ defines a rational map $\Phi:\bP^2 \dashrightarrow \bP^2$ whose base scheme is   $\Delta$.  Therefore $F_i\cap F_j= \Delta \cup \Sigma$, with   $\Delta\cap\Sigma=\emptyset$ and $\Sigma$ a general fibre of $\Phi$, by the choice of the $F_i$'s. Therefore $\Sigma$ consists of  $(d-1)^2-n$ distinct points.

    (II) To fix ideas assume $(i,j)=(1,2)$.
    Since clearly there is no curve of degree $d-3$  containing all the $(d-1)^2$ points,    the Cayley-Bacharach property implies that for every decomposition $F_1\cap F_2=\Delta\cup Z\cup Y$ with $\ell(Z)=\binom{d}{2}-n$, $\ell(Y)=\binom{d-1}{2}$, we have that $h^0(\I_{Z\cup\Delta}(d-2))=0$ implies 
 $h^0(\I_Y(d-3))=0$ (\cite{SR85}, Theorem III, p. 98).  
Therefore a decomposition $F_1\cap F_2=\Delta\cup Z\cup Y$ satisfies the conditions  stated in (II) if and only if  
$h^0(\I_{Z\cup\Delta}(d-2))=0$.

\noindent
By contradiction, suppose that such a $Z$   does not exists.
 Then for every $T\subset \Sigma$ such that $\ell(T)=\binom{d}{2}-n$ we have $h^0(\I_{T\cup\Delta}(d-2))>0$; choose   $T$    such that $h^0(\I_{T\cup\Delta}(d-2))>0$ is minimal. Since $T\cup\Delta$ does not impose independent conditions to the curves of degree $d-2$, there is $t\in T$ such that 
$$
h^0(\I_{T\cup\Delta}(d-2))=h^0(\I_{(T\cup\Delta)\setminus \{t\}}(d-2))
$$
Since $F_1\cap F_2$ is complete intersection of two curve of degree $d-1$, no curve of degree $d-2$ contains $F_1\cap F_2$. Therefore there is a point $z\in F_1\cap F_2$ not contained in   the base locus of the linear system $|H^0(\I_{T\cup\Delta}(d-2))|=|H^0(\I_{(T\cup\Delta)\setminus \{t\}}(d-2))|$. Thus we have
$$
h^0(\I_{[(T\cup\Delta)\setminus \{t\}]\cup \{z\}}(d-2))< h^0(\I_{T\cup\Delta}(d-2))
$$
contradicting the choice of $T$. 
\end{proof}

\emph{Proof of Proposition \ref{T:ffinal}}.
Let $\wt \sigma \in H^0(\bP^2,\I_\Delta(d))$ and let $D\subset \bP^2$ be the degree $d$ curve $\wt v=0$. Then diagram \eqref{E:pullbacksing2} fits into the following one:
\begin{equation}\label{E:pullbacksing3}
    \xymatrix{
    0\ar[r]&\olTC(d-3)\ar@{=}[d] \ar[r]&\F\ar[r]\ar[d]&\I_\Delta(d-3)\ar[r]\ar[d]^-{\wt \sigma}&0 \\
    0\ar[r]&\olTC(d-3) \ar[r]& \cO_{\bP^2}(d-2)\otimes V\ar[d] \ar[r]^-\partial&\I_\Delta(2d-3)\ar[r]\ar[d]&0 \\
    &&\cO_D(2d-3)(-\Delta)\ar@{=}[r]&\cO_D(2d-3)(-\Delta) 
    }   
   \end{equation}
   where $V=\langle F_X,F_Y,F_Z\rangle$. 
       In order to prove the proposition  we need  to   choose $\wt \sigma$ in such a way that 
        $$
        \wt \sigma\cdot H^0(\I_\Delta(d-3))\cap \mathrm{Im}(H^0(\partial))=(0)
        $$
        This will imply that   $H^0(\I_\Delta(d-3))\to H^1(\TC(d-3))$ is bijective, proving  the proposition.

     We choose a general basis of $V$ so that   $F_1\cap F_2$ consists of distinct points. This can be done, by Lemma \ref{L:Zexists2}. Moreover we choose $Z$ and $Y$ as in Lemma \ref{L:Zexists2}, and
       $\wt \sigma$ so that $D$  contains $\Delta\cup Z$ but no point of $Y$. 
       
       Note that $H^0(\olTC(d-3))=0$ because $\ol C$ is nodal and irreducible and therefore there are no syzygies of degree $\le d-2$ among $F_X,F_Y,F_Z$ (see \cite{DS14}, Example 2.2(i) and \cite{DS15}, Th. 4.1). In particular $H^0(\partial)$ is injective, so we will be free to identify $H^0(\cO_{\bP^2}(d-2))\otimes V$ with   $\mathrm{Im}(H^0(\partial))\subset H^0(\bP^2,\I_\Delta(2d-3))$.  
       More specifically, we identify $H^0(\bP^2,\cO(d-2))\otimes V$ with the vector subspace of $H^0(\bP^2,\I_\Delta(2d-3))$ consisting of elements  of the form $aF_1+bF_2+cF_3$, with $a,b,c\in H^0(\bP^2,\cO(d-2))$. Those for which $c=0$ are identified with $H^0(\bP^2,\cO(d-2))\otimes\langle F_1,F_2\rangle$, which in turn is  $H^0(\bP^2,\I_{\Delta\cup Z\cup Y}(2d-3))$,  since both have dimension $2\binom{d}{2}$.  We thus have the following commutative and exact diagram:
       \begin{small}
       
\begin{equation}
    \xymatrix{
    &H^0(\bP^2,\I_{\Delta\cup Z\cup Y}(2d-3))\ar[d]\ar[r]^-\alpha&H^0(D,\cO_D(2d-3)(-\Delta-Z))\ar[d] \\
    0\to H^0(\bP^2,\cO(d-3))\ar[r]^-{\wt v}\ar[d]&H^0(\bP^2,\I_\Delta(2d-3))\ar[r]^-\beta\ar[d]^-\gamma&H^0(D,\cO_D(2d-3)(-\Delta))\ar[r]\ar[d]&0 \\
    0\to H^0(\cO_Y(2d-3))\ar[r]&H^0(\cO_{Z\cup Y}(2d-3)\ar[r]&\cO_Z(2d-3)\ar[r]&0
    }
\end{equation}
      
       \end{small}
       Let $w:=aF_1+bF_2\in \ker(\alpha)$. Since $w$ vanishes on $D$, it is of the form $w=\wt \sigma h$, with $h\in H^0(\bP^2,\J_Y(d-3))$, because $\wt \sigma$ does not vanish on $Y$ but $w$ does. By  Lemma \ref{L:Zexists2}(II) we have $h=0$. Therefore $\alpha$ is injective. In fact $\alpha$ is bijective but we will not need it.

 Suppose  that  $0\ne u:=aF_1+bF_2+cF_3\in \wt \sigma\cdot H^0(\bP^2,\I_\Delta(d-3))$ for some $a,b,c\in H^0(\bP^2,\cO(d-2))$. Then $u\in H^0(\bP^2,\I_{2\Delta}(2d-3))$, i.e. it vanishes twice on $\Delta$.  Moreover, since $u\in \ker(\beta)$ and $\alpha$ is injective, we have $c\ne 0$. So $0\ne \gamma(u)=\gamma(cF_3)$, and $cF_3$ is non-zero on $Y$. Moreover $F_3$ vanishes simply on $\Delta$ and does not vanish on $Z$. Therefore $0\ne c\in H^0(\bP^2,\I_{\Delta\cup Z}(d-2))$. This   contradicts   Lemma \ref{L:Zexists2}(II) and proves Proposition \ref{T:ffinal} (and   Theorem \ref{T:nodalintro}). \hfill$\square$





\medskip\noindent
\textsc{Dipartimento di Matematica e Fisica \\ Università Roma Tre \\ L.go S.L. Murialdo, 1 -  00146 Roma, Italia.}\\
\texttt{sernesi@gmail.com}

\end{document}